\documentclass[12pt,reqno]{amsproc}
  \usepackage{latexsym} 
  \usepackage[all]{xy}
  \usepackage{amsfonts} 
  \usepackage{amsthm} 
  \usepackage{amsmath} 
  \usepackage{amssymb}
  \usepackage{pifont}  
  \usepackage{enumerate}
  \xyoption{2cell}

 \usepackage{tikz}
\usetikzlibrary{arrows}

  \def\<{{\langle}} 
  \def\>{{\rangle}}

  \def\note#1{{}}

  \def\note#1{} 

  \def\cO{{\mathcal O}}

  \def\beq{\begin{equation}} 
  \def\eeq{\end{equation}}

 \def\coker{\mathrm{coker}}

  \newcounter{zlist} 
  \newenvironment{zlist}{\begin{list}{(\arabic{zlist})}{ 
  \usecounter{zlist}\leftmargin2.5em\labelwidth2em\labelsep0.5em 
  \topsep0.6ex
  \parsep0.3ex plus0.2ex minus0.1ex}}{\end{list}}

  \newcounter{blist} 
  \newenvironment{blist}{\begin{list}{(\alph{blist})}{ 
  \usecounter{blist}\leftmargin2.5em\labelwidth2em\labelsep0.5em 
  \topsep0.6ex 
  \parsep0.3ex plus0.2ex minus0.1ex}}{\end{list}} 

  \newcounter{rlist}

\def\stac#1{\raise-.2cm\hbox{$\stackrel{\displaystyle\otimes}{\scriptscriptstyle{#1}}$}}

\def\cten#1{\raise-.2cm\hbox{$\stackrel{\displaystyle\widehat{\otimes}}
{\scriptscriptstyle{#1}}$}}

  \def\Label#1{\label{#1}\ifmmode\llap{[#1] }\else 
  \marginpar{\smash{\hbox{\tiny [#1]}}}\fi} 
  \def\Label{\label}

  \newtheorem{proposition}{Proposition}[section]
  \newtheorem{lemma}[proposition]{Lemma} 
   
  \newtheorem{theorem}[proposition]{Theorem} 

\theoremstyle{definition}

  \theoremstyle{remark} 
  \newtheorem{remark}[proposition]{Remark}

  \newcounter{c} 
   
  \newcommand{\etyk}[1]{\vspace{-7.4mm}$$\begin{equation}\Label{#1} 
  \addtocounter{c}{1}} 
  \renewcommand{\]}{\ifnum \value{c}=1 $$\else \end{equation}\fi} 
\def\KK{{\mathbb K}}

\def\NN{{\mathbb N}}

\newcommand{\Cc}{\mathcal{C}}

\newcommand{\Ii}{\mathcal{I}}

\def\*C{{}^*\hspace*{-1pt}{\Cc}}

\def\text#1{{\rm {\rm #1}}}

 \def\1{\mathbf{1}}

\def\bomega{\bar{\omega}}
\def\bnu{\bar{\nu}}

\begin{document}

\title[Differential smoothness of Hopf algebras]{Differential smoothness of affine Hopf algebras of Gelfand-Kirillov dimension two}

\author{Tomasz Brzezi\'nski}
 \address{ Department of Mathematics, Swansea University, 
  Swansea SA2 8PP, U.K.} 
  \email{T.Brzezinski@swansea.ac.uk}   
 
\begin{abstract}
Two-dimensional integrable differential calculi  for classes of Ore extensions of the polynomial ring and the Laurent polynomial ring in one variable are constructed. Thus it is concluded that all affine pointed Hopf domains of Gelfand-Kirillov dimension two which are not polynomial identity rings are differentially smooth.
\end{abstract}
\maketitle

\section{Introduction}
 An affine algebra $A$ is said to be {\em differentially smooth} if it admits an integrable differential calculus of dimension equal to the Gelfand-Kirillov dimension of $A$ (see Section~\ref{sec.pre} for definitions). The aim of this paper is to prove 
\begin{theorem}\label{thm.main}
Every affine pointed Hopf domain (over an algebraically closed field of characteristic 0) of Gelfand-Kirillov dimension two that is not a polynomial identity ring is differentially smooth.
\end{theorem}
Hopf algebras satisfying assumptions of Theorem~\ref{thm.main} are classified in \cite{WanZha:Hop} and are obtained by constructions described in \cite{GooZha:Noe}. The latter are examples of skew polynomial algebras or Ore extensions of the polynomial or Laurent polynomial ring in one variable. Therefore, along the way to Theorem~\ref{thm.main} we prove more generally that  members of a particular class of skew polynomial rings of Gelfand-Kirillov dimension two are differentially smooth. This is achieved by constructing explicitly two-dimensional integrable differential calculi over such rings.

\section{Preliminaries}
\label{sec.pre}\setcounter{equation}{0}
All algebras considered in this paper are associative and unital  over an algebraically closed field $\KK$ of characteristic 0.

Let $A$ be an algebra. By an {\em $n$-dimensional differential calculus over $A$} we mean a differential graded algebra $(\Omega A, d)$ such that
\begin{blist}
\item $\Omega A = \oplus_{k=0}^n \Omega^k A$, with $\Omega^0A=A$  and $\Omega^nA \neq 0$;
\item as an algebra, $\Omega A$ is generated by $A$ and $d(A)$;
\item $\ker d\mid_A = \KK.1$.
\end{blist}
Traditionally, the product of elements of $\Omega A$ of positive degree is denoted by $\wedge$.

An $n$-dimensional differential calculus over $A$ {\em admits a volume form}, say $\omega$, if $\Omega^nA$ is freely generated by $\omega$ as a right $A$-module. Associated to a volume form $\omega$ are the algebra automorphism $\nu_\omega: A\to A$ and the right $A$-module isomorphism $\pi_\omega: \Omega^n A\to A$, given by
\begin{equation}\label{nu.pi}
a\omega = \omega\nu_\omega(a), \qquad \pi_\omega(\omega a) = a, \qquad \mbox{for all }a\in A.
\end{equation}
We write $\Ii_k A$ for the space of right $A$-module maps $\Omega^kA\to A$. The direct sum $\Ii A :=\oplus_{k=0}^n \Ii_k A$ is a right $\Omega A$-module with multiplication
$$
(\varphi \cdot \omega')(\omega'') = \varphi (\omega '\wedge \omega''), \qquad \mbox{for all } \omega' \in \Omega^kA,\; \omega'' \in \Omega^m A,\;  \varphi\in \Ii_{k+m}A.
$$
Note that $\pi_\omega \in \Ii_n A$. A volume form $\omega \in \Omega^n A$ is said to be an {\em integrating form} if the left multiplication maps by $\pi_\omega$,
$$
\Omega^k A \to \Ii_{n-k}A, \qquad \omega'\mapsto \pi_\omega\wedge \omega', \qquad k=1,2,\ldots ,n-1,
$$
are bijective. A calculus admitting such a form is said to be {\em integrable}. An integrable calculus can be equivalently characterized by the existence of the bimodule complex $(\Ii_\bullet A,\nabla)$ (known as the {\em complex of integral forms}) isomorphic to the de Rham complex $(\Omega A, d)$. The boundary map $\nabla: \Ii_1 A\to A$ is a {\em divergence}, i.e.\  it satisfies the (right connection) Leibniz rule
$$
\nabla(\varphi \cdot a) = \nabla(\varphi) a + \varphi(da), \qquad \mbox{for all } a\in A,\ \varphi\in \Ii_1A.
$$
 The cokernel map $\Lambda: A \to \coker \nabla$ is called an {\em integral} associated to $\nabla$.
 
 An affine algebra of Gelfand-Kirillov dimension $n$ is said to be {\em differentially smooth} if it admits an $n$-dimensional integrable calculus. Thus in contrast  to other notions of smoothness of algebras  \cite{Sch:smo}, \cite{Van:rel}, which are more of homological nature, differential smoothness insists on existence of a particular differential structure of a specified dimension which admits a non-commutative version of the Hodge star isomorphism. 
Examples of differentially smooth algebras include the coordinate algebras of quantum groups such as $\cO(SL_q(2))$ or of quantum spaces such as the Podle\'s standard two-sphere or the Manin's plane. More surprisingly perhaps they also include coordinate algebras of classically non-smooth manifolds such as the pillow orbifold, cones or singular lens spaces \cite{BrzSit:smo}. It might be worth pointing out that all these examples are also homologically smooth in the sense of \cite{Van:rel}.

A characterization of differentially smooth algebras which will be of main usage in what follows is given in 

\begin{lemma}[Lemma~2.7 in \cite{BrzSit:smo}]\label{lem.integrating}
Let $\Omega A$ by an $n$-dimensional differential calculus over $A$ admitting a volume form $\omega$. Assume that, for all $k=1,2,\ldots, n-1$, there exist a finite number of forms $\omega_i^k, \bomega_i^k \in \Omega^kA$ such that, for all $\omega'\in \Omega^kA$,
\begin{equation}\label{dual.basis}
\omega' = \sum_i\omega_i^k \pi_\omega(\bomega_i^{n-k} \wedge \omega ') = \sum_i \nu_\omega^{-1}\left(\pi_\omega(\omega'\wedge  \omega_i^{n-k} )\right)\bomega_i^k,
\end{equation}
where $\pi_\omega$ and $\nu_\omega$ are defined by \eqref{nu.pi}. Then $\omega$ is an integrating form.
\end{lemma}

In the case of Lemma~\ref{lem.integrating}, the divergence is
\begin{equation}\label{hom.con.omega}
\nabla: \Ii_1 A\to A, \qquad \varphi \mapsto  (-1)^{n-1} \sum_i \pi_\omega \left( d\left(\nu_\omega^{-1}\left(\varphi(\omega_i^{1})\right)\right)\bomega_i^{n-1}\right).
\end{equation}

\section{Integrable differential calculi over skew polynomial rings}\label{sec.Ore}\setcounter{equation}{0}
Let $A$ be an algebra and $\sigma$ an automorphism of $A$. The linear map $\delta: A\to A$ is called a {\em $\sigma$-derivation}, provided for all $a,a'\in A$,
$$
\delta(aa') = \delta(a)a' + \sigma(a)\delta(a').
$$
An {\em Ore extension of $A$} or a {\em skew polynomial ring over $A$} associated to a $\sigma$-derivation $\delta$ is an extension of $A$ obtained by adjoining of a generator $y$ that is required to satisfy
$$
ya = \sigma(a) y + \delta(a), \qquad \mbox{for all } a\in A.
$$
Such an extension is denoted by $A[y;\sigma,\delta]$.

\subsection{Differentially smooth Ore extensions of the polynomial ring.}
All automorphisms $\sigma$ of the polynomial ring $\KK[x]$ have the form
\begin{equation}\label{sigma}
\sigma_{q,r} (x) =qx +r,
\end{equation}
where $q,r\in \KK$, $q\neq 0$. Furthermore, any element $p(x)\in \KK[x]$ determines a $\sigma_{q,r}$-derivation of $\KK[x]$ by 
\begin{equation}\label{delta}
\delta_p(f(x)) = \frac{f(\sigma_{q,r}(x)) - f(x)}{\sigma_{q,r} (x)  -x}\ p(x),
\end{equation}
where \eqref{delta} is to be understood as a suitable limit when $q=1$, $r=0$, i.e.\ when $\sigma_{q,r}$ is the identity map. The Ore extension $k[x][y;\sigma_{q,r}, \delta_{p}]$  will be denoted by $A[q,r ;p(x)]$. Thus $A[q,r;p(x)]$ is generated by $x,y$ subject to relation
\begin{equation}\label{rel}
yx = qxy + ry + p(x)
\end{equation}

\begin{lemma}\label{lem.homo}
Let
\begin{equation}\label{nu.start}
\nu_x(x) = x, \quad \nu_x(y) = qy + p'(x) \quad \mbox{and}\quad  \nu_y(x) = \sigma^{-1}_{q,r}(x), \quad \nu_y(y) = y,
\end{equation}
where $p'(x)$ is the $x$-derivative of $p(x)$.  
\begin{zlist} 
\item The symbols defined by \eqref{nu.start} simultaneously  extend to algebra automorphisms $\nu_x, \nu_y$  of $A[q,r;p(x)]$ only in the following three cases:
\begin{blist}
\item $q=1$, $r=0$ with no restriction on $p(x)$;
\item $q=1$, $r\neq 0$ and $p(x)=c$, $c\in \KK$;
\item $q\neq 1$, $p(x) = c(x+ \frac{r}{q-1})$, $c\in \KK$ with no restriction on $r$.
\end{blist}
\item In any of the cases (a)--(b) 
\begin{equation}\label{nu.nu}
\nu_y\circ \nu_x = \nu_x\circ \nu_y.
\end{equation}
\end{zlist}
\end{lemma}
\begin{proof}
(1) 
Clearly (a)--(c) exhaust all possible choices of $q$ and $r$, hence only restrictions on $p(x)$ need be studied in each case. The map $\nu_x$ can be extended to an algebra homomorphism if and only if the definitions of $\nu_x(x)$, $\nu_x(y)$ respect relation \eqref{rel}, i.e.
\begin{equation}\label{condition}
\nu_x(y)\nu_x(x) -\nu_x(qx +r)\nu_x(y) = p(\nu_x(x)).
\end{equation}
This yields a differential equation
\begin{equation}\label{condition.1}
((q-1)x +r)p'(x) = (q-1) p(x).
\end{equation}
If $q=1$ and $r=0$ both sides of \eqref{condition.1} are identically zero, hence there is no restriction on $p(x)$. If $q=1$ and $r\neq 0$, $p'(x)=0$, so $p(x)$ is a constant polynomial. Finally, if $q\neq 1$, by comparing coefficients of polynomials on both sides of \eqref{condition.1} one easily finds that $p(x) = c(x+ \frac{r}{q-1})$ for any $c\in \KK$. Thus $\nu_x$ is an algebra map precisely in one of the cases (a)--(c).

Condition \eqref{condition} for $\nu_y$, leads to the constraint 
\begin{equation}\label{condition.2}
p(\sigma^{-1}_{q,r}(x)) = \delta_p(\sigma^{-1}_{q,r}(x)).
\end{equation}
Since $\sigma^{-1}_{q,r}(x) = q^{-1}(x-r)$ and $\delta_p$ is a $\sigma_{q,r}$-derivation with $\delta_p(x) =p(x)$, \eqref{condition.2} is equivalent to
\begin{equation}\label{condition.3}
p(q^{-1}(x-r)) = q^{-1}p(x).
\end{equation}
Obviously, if $q=1$ and $r=0$, there are no restrictions on $p$. If $q=1$ and $r\neq 0$, then $p(x-r) =p(x)$, which implies that $p(x)$ is a constant polynomial. If $q\neq 1$, one easily checks that $p(x) = c(x+ \frac{r}{q-1})$ solves equation \eqref{condition.3}. This completes the proof that both $\nu_x$ and $\nu_y$ can be extended to algebra endomorphisms of $A[q,r;p(x)]$ if and only if one of the conditions (a)--(c) is satisfied. In all these cases, the inverses of $\nu_x$ and $\nu_y$ can be derived as
$$
\nu_x^{-1}(x) =x, \quad \nu_x^{-1}(y) = q^{-1}(y - p'(x)) \quad \mbox{and}\quad  \nu^{-1}_y(x) = \sigma_{q,r}(x), \quad \nu^{-1}_y(y) = y,
$$
thus completing the proof of the first assertion of the lemma.

(2) Since $\nu_x$, $\nu_y$ are algebra maps suffices it to check the equality \eqref{nu.nu} on the generators $x$, $y$. In the case (a), $\nu_y$ is the identity map, hence \eqref{nu.nu} is automatically satisfied. In the case (b),  $\nu_x$ is the identity map, hence again \eqref{nu.nu} is automatically satisfied. Finally, in the case (c) the equality \eqref{nu.nu} follows by the fact that $x$ is a fixed point of $\nu_x$ and $y$ is a fixed point of $\nu_y$, while the actions on the the other generators simply rescale and translate them by a constant.
\end{proof}

\begin{remark}
Note that if $q\neq 1$ is a root of unity, then \eqref{condition.3} has a richer space of solutions. Suppose $q^n=1$. Since the polynomial $x+ \frac{r}{q-1}$ satisfies  \eqref{condition.3}  so does any linear combination of  polynomials $(x+ \frac{r}{q-1})^{ln+1}$, $l\in \NN$.
\end{remark}

\begin{proposition}\label{prop.poly}
If $A[q,r;p(x)]$  satisfies one of the conditions (a)--(c) in Lemma~\ref{lem.homo}, then it is differentially smooth.
\end{proposition}
\begin{proof}
The algebras $A[q,r;p(x)]$ have Gelfand-Kirillov dimension two, hence a two-dimensional integrable calculus need be constructed. Let  $\Omega^1A[q,r;p(x)]$ be a free right $A[q,r;p(x)]$-module of rank two with generators $dx$, $dy$. Define the left $A[q,r;p(x)]$-module structure by 
\begin{equation}\label{rel.diff.gen}
adx = dx \nu_x(a), \qquad ady =  dy\nu_y(a), \qquad \mbox{for all } a\in A[q,r;p(x)],
\end{equation}
where $\nu_x$, $\nu_y$ are algebra automorphisms defined in Lemma~\ref{lem.homo}. Explicitly, in terms of generators, the relations in  $\Omega^1A[q,r;p(x)]$ come out as
\begin{equation}\label{rel.diff}
xdx = dx x, \quad xdy = q^{-1}dyx -q^{-1}rdy, \quad ydx = qdx y +dxp'(x), \quad ydy=dyy.
\end{equation}
We would like to extend $x\mapsto dx$, $y\mapsto dy$ to a map $d: A[q,r;p(x)]\to \Omega^1A[q,r;p(x)]$ satisfying the Leibniz rule. This is possible if the Leibniz rule is compatible with the only nontrivial relation \eqref{rel}, i.e.\ if
\begin{equation}\label{poly.Leibniz}
dyx +ydx = qdxy +qxdy + rdy + dp(x).
\end{equation}
Note that in view of the first of equations \eqref{rel.diff} which defines the usual commutative calculus on the polynomial ring $\KK[x]$, $dp(x) = dx p'(x)$. One easily checks using  \eqref{rel.diff} that equality 
\eqref{poly.Leibniz} is true. 

Define linear maps $\partial_x, \partial_y: A[q,r;p(x)] \to A[q,r;p(x)]$ by
\begin{equation} \label{partial}
d(a) = dx \partial_x(a) + dy \partial_y(a), \qquad \mbox{for all } a\in A[q,r;p(x)].
\end{equation}
These are well-defined since $dx$ and $dy$ are free generators of the right $A[q,r;p(x)]$-module $\Omega^1A[q,r;p(x)]$. By the same token $d(a) =0$ if and only if $\partial_x(a) =\partial_y(a) =0$. Using relations \eqref{rel.diff.gen}  and definitions of the maps $\nu_x$ and $\nu_y$ one easily finds that
\begin{equation}\label{partial.mn}
\partial_x(x^ky^l) = kx^{k-1}y^l, \qquad \partial_y(x^ky^l) = l\sigma_{q,r}(x^{k})y^{l-1} = l(qx+r)^{k}y^{l-1}.
\end{equation}
In particular, $\partial_x(\sum_{k,l} c_{kl}x^ky^l) =0$ if and only if $c_{k,l}=0$ whenever $(k,l)\neq (0,0)$. Thus $d(a) =0$ if and only if $a$ is a scalar multiple of the identity.

The universal extension of $d$ to higher forms  compatible with \eqref{rel.diff} gives the following rules for $\Omega^2A[q,r;p(x)]$
\begin{equation}\label{rel.diff.2}
dx\wedge dx = dy\wedge dy =0, \qquad dy\wedge dx = -qdx\wedge dy.
\end{equation}
Note that the last of the necessary equations \eqref{rel.diff.2} does not induce any additional constraints since for all $a\in A$,
\begin{eqnarray*}
a(dy\wedge dx + qdx\wedge dy) &=& dy\wedge dx \nu_x\circ\nu_y(a) + qdx\wedge dy \nu_y\circ\nu_x(a)\\
&=&  qdx\wedge dy \left(\nu_y\circ\nu_x(a) - \nu_x\circ\nu_y(a)\right) = 0,
\end{eqnarray*}
by \eqref{nu.nu}. Thus $\omega := dx\wedge dy$ freely generates $\Omega^2A[q,r;p(x)]$ as a right $A[q,r;p(x)]$-module. Furthermore,
$$
a\omega = \omega \nu_y(\nu_x(a)), \qquad \mbox{for all } a\in A[q,r;p(x)],
$$
and since both $\nu_x$ and $\nu_y$ are automorphisms, $\omega$ is a volume form and 
$\nu_\omega = \nu_y\circ \nu_x.$
This completes the construction of a two-dimensional differential calculus $\Omega A[q,r;p(x)]$ with a volume form.

Define 
\begin{equation}\label{omegas}
\omega_1 = dx, \quad \bomega_1 = -q^{-1}dy, \quad \mbox{and} \quad \omega_2 = dy, \quad \bomega_2 = dx.
\end{equation}
Then, for all $\omega' = dx a + dy b$,
\begin{eqnarray*}
\omega_1\pi_\omega(\bomega_1\wedge \omega') + \omega_2\pi_\omega(\bomega_2\wedge \omega') &=& dx\pi_\omega(-q^{-1}dy\wedge dxa) +  dy \pi_\omega(dx\wedge dyb) \\
&=& dx a + dy b = \omega'.
\end{eqnarray*}
Furthermore, using relations \eqref{rel.diff.gen} we can compute,
\begin{eqnarray*}
\sum_i \nu_\omega^{-1}\left(\pi_\omega(\omega'\wedge  \omega_i )\right)\bomega_i  &=& -q^{-1}\nu_\omega^{-1}\left(\pi_\omega(dyb\wedge dx)\right)dy+ \nu_\omega^{-1}\left(\pi_\omega(dxa\wedge dy\right)dx\\
&=& \nu_y^{-1}(b)dy +  \nu_y^{-1}\circ \nu_x^{-1}\circ  \nu_y(a)dx\\
&=& dy b + dx \nu_x\circ\nu_y^{-1}\circ \nu_x^{-1}\circ  \nu_y(a) =\omega',
\end{eqnarray*}
where the last equality follows by \eqref{nu.nu}. By Lemma~\ref{lem.integrating}, the calculus $\Omega A[q,r;p(x)]$ is integrable, and hence $A[q,r;p(x)]$ is a differentially smooth algebra as claimed.
\end{proof}

\begin{remark}\label{rem.div}
The maps $\varphi_x, \varphi_y: \Omega^1 A[q,r;p(x)]\to A[q,r;p(x)]$ defined by
$$
\varphi_x(dx a+ dyb) = a, \qquad \varphi_y(dx a+ dyb) =b,
$$
form a free basis for the module $\Ii_1 A[q,r;p(x)]$. An easy calculation reveals that 
\begin{equation}\label{div.zero}
\nabla(\varphi_x) = \nabla(\varphi_y) =0,
\end{equation}
where $\nabla$ is the divergence defined by \eqref{hom.con.omega}. Therefore, for all $a,b\in A[q,r;p(x)]$,
$$
\nabla(\varphi_x\cdot a +\varphi_y\cdot b) = \partial_x(a) +\partial_y(b),
$$
where $\partial_x$, $\partial_y$  are defined by  \eqref{partial}. It follows that $\nabla$ is surjective, hence the corresponding integral $\Lambda$ is identically zero. 

Directly by their construction $\partial_x$ and  $\partial_y$  are skew derivations (with corresponding automorphisms $\nu_x$, $\nu_y$), and $\Omega^1 A[q,r;p(x)]$ is the calculus induced by them in the sense of \cite{BrzElK:int}. Therefore, $\nabla$ is a unique divergence that satisfies condition \eqref{div.zero} by \cite[Theorem~3.4]{BrzElK:int}.
\end{remark}

\subsection{Differentially smooth Ore extensions of the Laurent polynomial ring.}
Automorphisms of the Laurent polynomial ring $\KK[x,x^{-1}]$ have the form
$$
\sigma_{q,\pm} (x) = qx^{\pm 1}, \qquad q\in \KK\setminus\{0\}.
$$
As in the case of the polynomial ring, $\sigma_{q,\pm}$-derivations are determined by their values at $x$, and we write $\delta_p$ for the derivation such that $\delta_p(x) =p(x)$. An Ore extension $\KK[x,x^{-1}][y;\sigma_{q,\pm}, \delta_p]$ is denoted by $A[q,\pm; p(x)]$. Explicitly, $A[q,\pm; p(x)]$ is generated by $x$, its inverse $x^{-1}$ and by $y$ such that
\begin{equation}\label{rel.Laurent}
yx = q x^{\pm 1}y + p(x).
\end{equation}

\begin{lemma}\label{lem.homo.Laurent}
Define
\begin{equation}
\bnu_x(x) = x, \quad \bnu_x(y) = -qx^{-2}y + p'(x) \quad \mbox{and}\quad  \bnu_y(x) = qx^{-1}, \quad \bnu_y(y) = y,
\end{equation}
Then both $\bnu_x$, $\bnu_y$ simultaneously extend to algebra automorphisms of $A[q,-;p(x)]$ if and only if
\begin{equation}\label{laurent.poly}
p(x) = c(x-qx^{-1}), \qquad c\in \KK.
\end{equation}
Furthermore, the resulting automorphisms satisfy the following equality
\begin{equation}\label{nu.nu.bar}
\bnu_y\circ \bnu_x(a)x^2 = x^2\bnu_x\circ \bnu_y(a), \qquad \mbox{for all } a\in A[q,-;c(x-qx^{-1})].
\end{equation}
\end{lemma}
\begin{proof}
Arguing as in the proof of Lemma~\ref{lem.homo}, $\bnu_x$ can be extended to an algebra homomorphism provided
$$
\bnu_x(y)\bnu_x(x) = q \bnu_x(x)^{-1}\bnu_x(y) + p(\bnu_x(x)).
$$
This leads to the differential equation
$$
(x-qx^{-1})p'(x) = (1+qx^{-2})p(x),
$$
with polynomials \eqref{laurent.poly} as only solutions. On the other hand, one easily checks that
$$
\bnu_y(y)\bnu_y(x) = q \bnu_y(x)^{-1}\bnu_y(y) + c(\bnu_y(x) - q\bnu_y(x)^{-1}).
$$
Thus also $\bnu_y$ can be extend to an algebra endomorphism. Clearly, $\bnu_y$ is then an automorphism. The inverse of $\bnu_x$ is determined from $\bnu_x^{-1}(y) = q^{-1}(c(x^2+q) -x^2y)$.

Since $\bnu_x$ and $\bnu_y$ are algebra maps, the equality \eqref{nu.nu.bar} needs only to be checked for $a=x,y$. The first case is trivial, in the second case
$$
\bnu_y\circ \bnu_x(y)x^2 =-q^{-1}x^2yx^2 + c(1+q^{-1}x^2)x^2 = -qy +c(x^2 +q) = x^2 \bnu_x\circ \bnu_y(y),
$$
where the middle equality follows by a repeated use of \eqref{rel.Laurent} with $p(x) = c(x-qx^{-1})$. This completes the proof of the lemma.
\end{proof}

\begin{remark}
More generally $\bnu_y$ alone extends to an algebra automorphism whenever $p(x) = \sum_{i=1}^na_i(x^i-q^ix^{-i})$.
\end{remark}

\begin{proposition}\label{prop.Laurent}
Algebras $A[1,+;p(x)]$, $A[q,+; c]$ and $A[q,-;c(x-qx^{-1})]$ are differentially smooth.
\end{proposition}
\begin{proof}
All these algebras have Gelfand-Kirillov dimension two, so two-dimensional integrable calculi need be constructed. The first two algebras are localisations of $A[q,0;p(x)]$, and localisations of differential calculi described in the proof of Proposition~\ref{prop.poly} yield integrable two-dimensional calculi, hence $A[1,+;p(x)]$ and  $A[q,+; cx]$ are differentially smooth. An integrable calculus $\Omega A[q,-;c(x-qx^{-1})]$ over $A[q,-;c(x-qx^{-1})]$ is defined as follows.

$\Omega^1A[q,-;c(x-qx^{-1})]$ is a right $A[q,-;c(x-qx^{-1})]$-module freely generated by $dx$ and $dy$. The left module structure is defined by
\begin{equation}\label{rel.diff.gen.Laurent}
adx = dx \bnu_x(a), \qquad ady =  dy\bnu_y(a), \qquad \mbox{for all } a\in A[q,-;c(x-qx^{-1})].
\end{equation}
This bimodule extends to a graded algebra 
generated by $x$, $y$, $dx$, $dy$ subject to the following relations
\begin{equation}\label{rel.diff.Laurent}
xdx = dx x, \quad xdy = qdyx^{-1}, \quad ydx = -qdx x^{-2}y +cdx(1+qx^{-2}), \quad ydy=dyy.
\end{equation}
\begin{equation}\label{rel.diff.Laurent.2}
dx\wedge dx = dy\wedge dy =0, \qquad dx\wedge dy = qdy\wedge dx x^{-2}.
\end{equation}
With these definitions the assignment $x\mapsto dx$, $y\mapsto dy$ can be extended to an exterior differential, thus yielding a  differential graded algebra over $A[q,-;c(x-qx^{-1})]$.

Linear endomorphisms $\partial_x$, $\partial_y$ of $A[q,-;c(x-qx^{-1})]$ can be defined by the formulae analogous to \eqref{partial}. In particular $\partial_x$ will have exactly the form $\eqref{partial.mn}$ with the only difference that $k$ is an integer, and the same argument as in the proof of Proposition~\ref{prop.poly} affirms that $d(a) =0$ only if $a$ is a scalar multiple of the identity. Thus  $\Omega A[q,-;c(x-qx^{-1})]$ is a calculus over $A[q,-;c(x-qx^{-1})]$.

The last of  equations \eqref{rel.diff.Laurent.2} does not induce any additional constraints since for all $a\in A$,
\begin{eqnarray*}
a(dx\wedge dy - qdy\wedge dx x^{-2}) &=& dx\wedge dy \bnu_y\circ\bnu_x(a) - qdy\wedge dx \bnu_x\circ\bnu_y(a) x^{-2}\\
&=&  -qdy\wedge dx \left(x^{-2}\bnu_x\circ\bnu_y(a) - \bnu_y\circ\bnu_x(a)x^{-2}\right) = 0,
\end{eqnarray*}
by \eqref{nu.nu.bar}. Thus the module $\Omega^2 A[q,-;c(x-qx^{-1})]$ is freely generated by $\bomega = dy\wedge dx$. This is a volume form with the corresponding automorphism $\nu_{\bomega} = \bnu_x\circ \bnu_y$. 
Finally, setting
$$
\omega_1 = dx, \quad \bomega_1 = dy, \quad \mbox{and} \quad \omega_2 = dy, \quad \bomega_2 = qdxx^{-2},
$$
one can verify that the requirements of Lemma~\ref{lem.integrating} are fulfilled (in proving of the second of equalities \eqref{dual.basis}, the relation \eqref{nu.nu.bar} plays a crucial role).
\end{proof}

\subsection{Proof of Theorem~\ref{thm.main}.}
By \cite{WanZha:Hop} if $H$ is a pointed affine Hopf domain of Gelfand-Kirillov dimension 2 that does not satisfy a polynomial identity then it falls into the following three classes:
\begin{blist}
\item $H$ is the universal enveloping algebra of the two-dimensional solvable Lie algebra, thus it is an algebra of type $A[1,0;x]$ with both $x$ and $y$ primitive elements.
\item $H$ is isomorphic to the algebra $A[q,+;0]$, $q\neq 1$, with $x$ a grouplike element and the coproduct $\Delta(y) = y\otimes 1 + x^n\otimes y$.
\item $H$ is isomorphic to the algebra $A[1,+;x^n-x]$, with $x$ a grouplike element and the coproduct $\Delta(y) = y\otimes x^{n-1} + 1\otimes y$
\end{blist}
All these algebras satisfy conditions in Lemma~\ref{lem.homo} and Lemma~\ref{lem.homo.Laurent}, and hence they are differentially smooth by Proposition~\ref{prop.poly} and Proposition~\ref{prop.Laurent}.


\begin{thebibliography}{99}{} 

\bibitem{BrzElK:int} T.\hspace{3pt}Brzezi\'nski, L.\hspace{3pt}El Kaoutit \& C.\hspace{3pt}Lomp, {\em Non-commutative integral forms and twisted multi-derivations}, J.\hspace{3pt}Noncommut.\hspace{3pt}Geom.\hspace{3pt} {\bf 4} (2010), 281--312. 

\bibitem{BrzSit:smo} T. Brzezi\'nski \& A.\ Sitarz, {\em Smooth geometry of the noncommutative pillow, cones and lens spaces}, arXiv:1410.6587 (2014).

\bibitem{GooZha:Noe} K.R.\ Goodearl \& J.J.\ Zhang, {\em Noetherian Hopf algebra domains of Gelfand-Kirillov dimension two}, J.\ Algebra {\bf 324} (2010), 3131--3168.

\bibitem{Sch:smo} W.F.\ Schelter, \emph{Smooth algebras}, J.\ Algebra {\bf 103}
 (1986), 677--685.

\bibitem{Van:rel} M.\ Van den Bergh, {\em A relation between Hochschild homology and cohomology for Gorenstein rings}, Proc.\
Amer.\ Math.\ Soc.\ {\bf 126} (1998), 1345--1348; Erratum: Proc.\ Amer.\ Math.\ Soc.\ {\bf 130} (2002), 2809--2810.

\bibitem{WanZha:Hop} D.-G. Wang,  J.J.\ Zhang \& G.-B.\ Zhuang, {\em Hopf algebras  of GK dimension two with vanishing Ext-group}, J.\ Algebra {\bf 388} (2013), 219--247.

\end{thebibliography}
\end{document}